\documentclass[11pt]{article}
\usepackage[english]{babel}
\usepackage{amsfonts}
\usepackage[utf8]{inputenc}
\usepackage{amsthm}
\usepackage{ mathrsfs }
\usepackage{ amsmath }
\usepackage{mathrsfs}
\usepackage{enumerate}
\usepackage[noadjust]{cite}
\usepackage{bbm}
\usepackage{makeidx}
\usepackage{graphicx}
\usepackage{frontespizio}
\usepackage{color}
\usepackage[makeroom]{cancel}
\usepackage[normalem]{ulem}
\usepackage{ amssymb }
\usepackage[a4paper, total={6in, 9in}]{geometry}
\usepackage{hyperref}
\usepackage{mathtools}
\usepackage{enumitem}
\usepackage[toc,page]{appendix}
\mathtoolsset{showonlyrefs=true}

\theoremstyle{plain}
\newtheorem{theorem}{Theorem}[section]
\newtheorem{lemma}[theorem]{Lemma}
\newtheorem{prop}[theorem]{Proposition}
\newtheorem{theoremx}{Theorem}

\newcommand{\X}{{\rm X}}

\theoremstyle{definition}
\newtheorem{definition}[theorem]{Definition}

\DeclareMathOperator\rr{\mathbb{R}}
\DeclareMathOperator\eps{\varepsilon}

\makeindex
\numberwithin{equation}{section}

\newcommand{\sfd}{{\sf d}}

\renewcommand{\a}{{\sf a}}
\renewcommand{\b}{{\sf b}}

\newcommand{\nn}{\mathbb{N}}

\title{A remark on two notions of flatness for sets in the Euclidean space}

\begin{document}

\author{Ivan Yuri Violo
	\thanks{SISSA, iviolo@sissa.it}}

\maketitle		

\begin{abstract}
	In this  note we compare two ways of measuring the $n$-dimensional ``flatness" of a set $S\subset \rr^d$, where $n\in \mathbb{N}$ and $d>n$. The first one is to consider the classical  Reifenberg-flat numbers $\alpha(x,r)$ ($x \in S$, $r>0$), which measure the minimal scaling-invariant Hausdorff distances  in $B_r(x)$ between $S$  and $n$-dimensional affine subspaces of $\rr^d$. The second is an `intrinsic' approach in which we  view the same set $S$ as a metric space (endowed with the induced Euclidean distance). Then we consider numbers $\a(x,r)$'s, that are the scaling-invariant  Gromov-Hausdorff distances between  balls centered at $x$ of radius $r$ in $S$ and the $n$-dimensional Euclidean ball of the same radius. 
	
	As main result of our analysis we make rigorous a phenomenon, first noted by David  and Toro, for which the numbers $\a(x,r)$'s behaves as the square of the numbers $\alpha(x,r)$'s. Moreover we show how this result finds application in extending the Cheeger-Colding intrinsic-Reifenberg theorem to the biLipschitz case.
	
	As a by-product of our arguments, we deduce  analogous results also for the Jones' numbers $\beta$'s (i.e.  the one-sided version of the numbers $\alpha$'s). 

\end{abstract}

\tableofcontents

\section{Introduction and main results}
In this short note we consider two ways of measuring the ``flatness" of a set in the Euclidean space. The first one is by considering its best approximation by affine planes: more precisely, given a set $S\subset  \rr^d$ and $n \in \mathbb{N}$, with $n < d$,  one defines
\begin{equation}\label{eq:alfadef}
	\alpha(x,r)\coloneqq r^{-1}\inf_{\Gamma}  \sfd_H(S\cap B_{r}(x),\Gamma\cap B_{r}(x)), \quad \text{for every $r>0$ and $x \in S$},\\
\end{equation}
where  $ \sfd_H$ is the Hausdorff distance and where the infimum is  taken among all the  $n$-dimensional affine planes $\Gamma$ containing $x$.

The second more metric-oriented approach is to use the \emph{Gromov-Hausdorff distance}, in particular given $S\subset  \rr^d$ and $n \in \mathbb{N}$, with $n < d$, we set
\[
	\a(x,r)\coloneqq r^{-1}\sfd_{GH}(B^{(S,\sfd_{Eucl})}_r(x),B^{\rr^n}_r(0)), \quad \text{for every  $r>0$ and $x \in  S,$ }
\]
where $(S,\sfd_{Eucl})$ is the metric space obtained by endowing $S$ with the  Euclidean distance. It follows  immediately from the definition of $\sfd_{GH}$ that
\begin{equation}\label{eq:trivial bound}
	\a(x,r)\le\alpha(x,r), \quad \text{for every  $r>0$ and $x \in  S$. }
\end{equation}
Moreover it is easy to build many examples for which
\[
\alpha(x,r)\le 4	\a(x,r),
\]
 with $	\a(x,r)\neq 0$ and arbitrary small (one can take $S$ to be a segment with a very small interval removed from its center). This shows that in general \eqref{eq:trivial bound} cannot be improved and leads to the intuition the the quantities $\alpha(x,r)$ and $\a(x,r)$ are in some sense equivalent. However there are  non-trivial cases in which the  stronger inequality
\[
\a(x,r)\le 100\alpha(x,r)^2
\]
holds. The key example  is the one of a very thin triangle: let $P,Q \in \rr^2$ be the two points in the upper half plane that are at distance 1 from the origin and at distance $\eps\in(0,1/2)$ from the $x$-axis and let $S \subset  \rr^2$  be the union of the two closed segments joining P and Q to the origin $O$. It can be immediately seen that $\alpha(O,1) \ge \eps$, while  projecting $S$ orthogonally onto the $x$-axis easily shows that $\a(O,r) \le 4\eps^2$.  

The aim of this note is to explore  the above phenomenon, that is to clarify to which extent and in which cases the quantities $\a(x,r)$ behave like the square of the quantities $\alpha(x,r).$ 

To state our main result we need the following notation: fixed $\eps\in(0,1/2)$, $S \subset \rr^d$ and $n \in \mathbb{N}$ with $n<d$, for every $i \in \mathbb{Z}$  we set
\begin{equation}\label{def:ai}
\begin{split}
	\alpha_i\coloneqq\sup_{x \in S\cap B_1(0)}\alpha(x,2^{-i}),\quad \a_i\coloneqq\sup_{x \in S\cap B_{1-\eps}(0)}\a(x,2^{-i}),
\end{split}
\end{equation}
where we neglected the dependence on $\eps$, $n$ and $S$. Our main result reads as follows:
\begin{theoremx}\label{thm:mainA}
	For every $n \in \mathbb{N}$ there exists $\delta(n)>0$ such that the following holds. Let $S\subset \rr^d$, with $d>n$,  $\eps\in(0,1/2)$ and define the numbers $\alpha_i,\a_i$ as in \eqref{def:ai}. Suppose that $\alpha_i\le \delta$ for every $i \ge \bar i-2$, for some $\bar i \in \nn$ with $2^{-\bar i}<\eps$, then
	\begin{equation}\label{eq:mainA precise}
		\sum_{i \ge \bar i} \a_i^{\lambda} <C_\lambda \sum_{i \ge \bar i-2} \alpha_i^{2\lambda}, \quad \forall \lambda>0,
	\end{equation}
where $C_\lambda$ is a positive constant depending only on $\lambda$ and $n$. In particular for every $\lambda>0$ it holds that
	\begin{equation}\label{eq:mainA}
	 \sum_{i \ge 0} \alpha_i^{2\lambda} < +\infty\Longrightarrow \sum_{i \ge 0} \a_i^\lambda < +\infty.
	\end{equation}
\end{theoremx}

Theorem \ref{thm:mainA}  will follow from a `weak' version of the inequality $\a(x,r)\lesssim\alpha(x,r)^2$ (see Theorem \ref{thm:precise}), which as said above cannot hold in its `strong' form. 

It has to be said that the fact that the numbers $\a(x,r)$'s ``behaves" as the square of numbers $\alpha(x,r)$'s  was already noted, at least at an informal level, by David and Toro (see \cite{snow}). However, to the author's best knowledge, both  the statement and the proof of Theorem \ref{thm:mainA} are new.

\bigskip
\begin{center}
	{\bf The Jones' numbers $\beta $}
\end{center}

We will also prove the analogue of Theorem \ref{thm:mainA} for the ``one sided"-version of the numbers $\alpha(x,r)$: given a set $S\subset  \rr^d$ and $n \in \mathbb{N}$, with $n <d$, we set
\begin{equation}\label{eq:betadef}
	\beta(r,x)\coloneqq  r^{-1} \,\inf_{\Gamma }  \sup_{y \in S\cap B_r(x)} \sfd(y,\Gamma), \quad \text{for every  $r>0$ and $x \in  S,$ }
\end{equation}
where  the infimum is taken among all the $n$-dimensional affine planes $\Gamma$ containing $x$ . The numbers $\beta(x,r)$ are usually refer to as $L^\infty$-\emph{Jones' numbers} (see for example \cite{holes}, \cite{bj} and \cite{jones}). It is immediate from the definition that $\beta(x,r)\le \alpha(x,r),$ for every $x\in S$ and $r>0.$

We then define the following metric analogue of $\beta(x,r)$: for a set $S\subset  \rr^d$ and $n \in \mathbb{N}$, with $n <d$,  we set
	\begin{equation*}
	\b(r,x)\coloneqq  r^{-1} \,\inf \left \{\delta \ : \ \text{there exists a $\delta$-isometry } f : (S\cap B^{\rr^d }_r(x),\sfd_{Eucl}) \to (B^{\rr^n}_{r}(0),\sfd_{Eucl})\right \},
\end{equation*}
for every  $r>0$ and $x \in  S$ (see Section \ref{sec:proofmain} for the definition of $\delta$-isometry). 
As for the numbers $\alpha(x,r)$ and $\a(x,r)$ we have the immediate inequality
\begin{equation}\label{eq:trivial bound 2}
	\b(x,r)\le2\beta(x,r), \quad \text{for every  $r>0$ and $x \in  S$ }.
\end{equation}
Similarly to the numbers $\alpha_i, \a_i$ we define for a given $S \subset \rr^d$, $n \in \mathbb{N}$ with $n<d$, a  fixed $\eps\in(0,1/2)$ and every $i \in \mathbb{Z}$
\begin{equation}\label{eq:bi}
\begin{split}
	\beta_i\coloneqq\sup_{x \in S\cap B_1(0)}\beta(x,2^{-i}),\quad \b_i\coloneqq\sup_{x \in S\cap B_{1-\eps}(0)}\b(x,2^{-i}),
\end{split}
\end{equation}
where we neglected the dependence on $\eps$, $n$ and $S$. Then we can prove the following:
\begin{theoremx}\label{thm:mainB}
	For every $n \in \mathbb{N}$ there exists $\delta(n)>0$ such that the following holds.  Let $S\subset \rr^d$, with $d>n,$ $\eps\in(0,1/2)$ and define the numbers $\beta_i,\b_i$ as in \eqref{eq:bi}.  Suppose that $\alpha_i\le \delta$ for every $i \ge \bar i-2$, for some $\bar i \in \nn$ with $2^{-\bar i}<\eps$ (where $\alpha_i$ are as in \eqref{def:ai}), then
\begin{equation}\label{eq:mainB precise}
	\sum_{i \ge \bar i} \b_i^{\lambda} <C_\lambda \sum_{i \ge \bar i-2} \beta_i^{2\lambda}, \quad \forall \lambda>0,
\end{equation}
where $C_\lambda$ is a positive constant depending only on $\lambda$ and $n$.
\end{theoremx}
In particular Theorem \ref{thm:mainB} implies that,  whenever $\limsup_{i\to +\infty} \alpha_i<\delta$ (with $\delta$ as in the statement of the theorem), we have that
\begin{equation}\label{eq:mainB}
	\sum_{i \ge 0} \beta_i^{2\lambda} < +\infty\Longrightarrow \sum_{i \ge 0} \b_i^\lambda < +\infty, 
\end{equation}
for every $\lambda>0$.

As for Theorem \ref{thm:mainA}, Theorem \ref{thm:mainB}  will be deduce from a `weak' version of the inequality $\b(x,r)\lesssim\beta(x,r)^2$, which is also contained in  Theorem \ref{thm:precise}.

\bigskip
\begin{center}
	{\bf Converse inequalites}
\end{center}
In the last section we will prove that, contrary to the inequality `$\a(x,r)\lesssim\alpha^2(x,r)$', the opposite estimate $\alpha^2(x,r)\lesssim \a(x,r)$ holds in full generality. Moreover we will also prove that   $\beta(x,r)^2\le \b(x,r)$ holds, provided $B_r(x)\cap S$ contains $n$  `sufficiently'  independent points. This combined with the results in Theorem \ref{thm:mainA} and \ref{thm:mainB} shows that the intrinsic `Gromov-Hausdorff' approach and the `extrinsic' Hausdorff approach to measuring `flatness' in the Euclidean space are in a sense equivalent up to a square factor.

\bigskip
\begin{center}
	{\bf Motivations and application to Reifenberg's theorem}
\end{center}

We now explain  the role, consequences  and  motivations of the results in this note.

It is essential to recall that the quantities $\alpha(x,r)$ and $\beta(x,r)$ coupled with smallness or summability conditions as in \eqref{eq:mainA}, \eqref{eq:mainB} (also called Dini-conditions) are tightly linked to parametrization and rectifiability results for sets in the Euclidean space. The more classical are the celebrated Reifenberg theorem (\cite{r}) and the rectifiability results of Jones (\cite{jones}), but there are also more recent and sophisticated  works containing variants, generalizations and refinements of these type of statements (see for example  \cite{toro}, \cite{snow}, \cite{holes} and the references therein). It also worth to mention the the works in \cite{nv}, \cite{env} and \cite{tolsa}, which contain similar results, but where $L^p$-versions of the $\beta$-Jones' numbers are considered.

There has been recently a growing interest in extending  statements for sets in $\rr^d$  as above (or  rectifiability results in general), to the setting of metric spaces. The most notable instance of this is the intrinsic-Reifenberg theorem by Cheeger and Colding (\cite{CC}), which has recently  found many applications especially in the theory of singular metric spaces with synthetic curvature conditions (see for example \cite{CC}, \cite{mk} and \cite{ricci2}).

If one is interested in extending to the setting of metric spaces  results in $\rr^d$  involving the quantities $\alpha(x,r)$ and $\beta(x,r)$ (or variants of them), it is more convenient to consider instead  the numbers $\a(x,r)$ and $\b(x,r)$. This is because the  the quantities $\alpha(x,r), \beta(x,r)$ are confined to the Euclidean space, while their ``Gromov-Hausdorff" counterparts are immediately generalized to the metric setting. For this reason it is essential to have a good understanding of the relation between the numbers $\alpha(x,r),\beta(x,r)$ and the numbers $\a(x,r),\b(x,r)$. This is the point in which  Theorem \ref{thm:mainA} and Theorem \ref{thm:mainB} find their relevance.  

Indeed our Theorem \ref{thm:mainA} shows that it makes sense to interpret the numbers $\alpha(x,r)$ as the square of $\a(x,r)$, at least when one is interested in their decay behaviour. To explain  further the consequence of this fact we now show that Theorem \ref{thm:mainA} is crucial if one wants to extend the biLipschitz version of Reifenberg theorem to the metric setting. Let us also say that this problem is what originated the writing of this note.

We first need to recall the classical Reifenberg's theorem:
\begin{theorem}[\cite{r}]\label{thm:reif}
	For every $n,d\in \nn$  with $n<d$ there exists $\delta=\delta(n,d)$ such that the following holds. Let $S\subset \rr^d$ be closed, containing the origin and such that
	\[
	\alpha_i<\delta,\quad \forall i \in \nn_0,
	\]
	where $\alpha_i$ are as in \eqref{def:ai}. Then there exists a biH\"older homeomorphism $F: \Omega\to S\cap B_{1/2}^{\rr^n}(0)$, where $\Omega$ is an open set in $\rr^n.$
\end{theorem}
It was proven by Toro that if we require, besides smallness, also a fast decay of the number $\alpha_i$ as $i \to +\infty,$ the biH\"older regularity of the map $F$ can be improved to biLipschitz. In particular we have the following:
\begin{theorem}[\cite{toro}]\label{thm:lipreif}
	For every $\eps>0$, $n,d\in \nn$  with $n<d$, there exists $\delta=\delta(n,d,\eps)>0$ such that the following holds. Let $S\subset \rr^d$ be closed, containing the origin and such that
	\begin{equation}\label{eq:dini1}
		\sum_{i \ge 0}\alpha_i^2<\delta,
	\end{equation}
	where $\alpha_i$ are as in \eqref{def:ai}. Then there exists a (1+$\eps$)-biLipschitz homeomorphism $F: \Omega\to S\cap B_{1/2}^{\rr^n}(0)$, where $\Omega$ is an open set in $\rr^n.$
\end{theorem}
It is a remarkable result by Cheeger and Colding  that Reifenberg's theorem can be generalized to metric spaces:
\begin{theorem}[\cite{CC}]\label{thm:CCreif}
	For every $n \in \nn$ there exists $\eps=\eps(n)>0$ such that the following holds. Let  $(Z,\sfd)$ be a complete metric space, let $z_0 \in Z$ and define $\a_i\coloneqq \sup_{z \in B_{2/3}(z_0)} 2^i\sfd_{GH}(B_{2^{-i}}(z),B^{\rr^n}_{2^{-i}}(0))$, $i \in \nn_0.$ Suppose that
	\[
	\a_i\le \eps, \quad \forall \, i \in \nn_0.
	\]
	Then there exists a biH\"older homeomorphism  $F: \Omega\to  B_{1/2}(z_0)$, where $\Omega$ is an open set in $\rr^n.$
\end{theorem}
As said above this result found recently a wide range of applications, in particular in the study of regularity of singular metric spaces.  
It is therefore natural to ask weather also an analogous of Theorem \ref{thm:lipreif} holds in the metric setting.
A careful analysis of the arguments in\cite{CC} shows that, with little  modifications, they can be adapted to prove  the following biLipschitz version of Theorem \ref{thm:CCreif}: 
\begin{theorem}\label{thm:CClip}
	For every $n \in \nn$ and $\eps>0$ there exists $\delta=\delta(n,\eps)>0$ such that the following holds. Under the notations and assumptions of  Theorem \ref{thm:CCreif}  suppose that
	\begin{equation}\label{eq:dini2}
		\sum_{i\ge0} \a_i<\delta. 
	\end{equation}
	Then there exists a (1+$\eps$)-biLipschitz homeomorphism  $F: \Omega\to  B_{1/2}(z_0)$, where $\Omega$ is an open set in $\rr^n.$
\end{theorem}
Comparing the above with Theorem \ref{thm:lipreif}, the presence of the summability assumption (which is stronger than square summability) might lead to think that something stronger than Theorem \ref{thm:CClip} should hold, at least for `nicer' metric spaces, like subsets of the Euclidean space. Indeed if one restricts its attention only to \eqref{eq:trivial bound} (which as we said, cannot be improved), the summability assumption \eqref{eq:dini2} for a subset $S$ of the Euclidean space (when regarded as metric space $(S,\sfd_{Eucl}))$ appears stronger than \eqref{eq:dini1}. Therefore it may seem that Theorem \ref{thm:lipreif} is in a sense missing some of the informations contained in the theorem of Toro. 

 However the key observation is that Theorem \ref{thm:mainA}  implies that
\begin{equation}\label{eq:stronger}
	\begin{split}
		\text{Theorem \ref{thm:CClip} is {\bf{stronger}} than Theorem \ref{thm:lipreif}},
	\end{split}
\end{equation}
where by ``stronger" we mean that  any set $S$ satisfying the hypotheses of Theorem \ref{thm:lipreif} also satisfies the hypotheses of Thoerem \ref{thm:CClip} (when regarded as metric space $(S,\sfd_{Eucl}))$).

Finally  Theorem \ref{thm:mainA} says also something about the sharpness of Theorem \ref{thm:CClip}, indeed it is well known that the power two in \eqref{eq:dini1} cannot be replaced by any higher order power (see for example \cite{toro}), in particular \eqref{eq:mainA} implies that also the power one in  \eqref{eq:dini2}  cannot be improved. This observation together with \eqref{eq:stronger} suggests that Theorem \ref{thm:CClip} is in a sense the correct generalization of Theorem \ref{thm:lipreif}.

It is worth to mention that another instance (besides Theorem \ref{thm:CClip}) where a summability condition on the Gromov-Hausdorff distances is natural and necessary in metric spaces was already observed by Colding (see \cite[Sec. 4.5]{colding}). Roughly said he proves that the summability of the Gromov-Hausdorff distance from a cone on dyadic scales on a Riemannian manifold is necessary and sufficient to have uniqueness of the tangent cone. Moreover he points out the  discrepancy between this summability assumption in comparison with the square summability assumption in  Theorem  \ref{thm:lipreif} by Toro. As for the biLipschitz Reifenberg above, our Theorem \ref{thm:mainA} explains this discrepancy.

\section*{Acknowledgements}
I am grateful to Guido De Philippis and Nicola Gigli for bringing this problem to my attention and for stimulating discussions. I also wish to thank Tatiana Toro and David Guy for their feedback and comments during the preparation of this note.

\section{Preliminary results}
We gather in this section some elementary and well known results that will be needed in the sequel. 

In what follows and in all this note we denote by $\sfd_H$ the Hausdorff distance  between sets  in $\rr^d$. Moreover  given an affine plane $\Gamma$ in $\rr^d$ and a point $x \in\rr^d$ we denote by $\sfd(x,\Gamma)$ the distance between $x$ and $\Gamma.$

The following elementary Lemma is well known in literature (see for example \cite[Lemma 12.62]{holes}).
\begin{lemma}\label{lem:frame close}
	For every $n \in \mathbb{N} $ with $n$ there exists a constant $C=C(n)>0$ such that the following holds. Suppose that $\Gamma_1,\Gamma_2$ are two affine $n$-dimensional planes in $\rr^d$, with $d>n$, such that there exist points $\{x_i\}_{i=0}^n\subset \Gamma_2\cap B^{\rr^d}_1(0)$ satisfying 
	\[
	\begin{split}
		&|x_i-x_0-e_i|\le \frac{1}{10}, \quad \text{for every $i=1,...,n$},\\
		&\sfd(x_i,\Gamma_1)\le \eps, \quad \text{for every $i=0,...,n$},
	\end{split}
	\]
	where $e_1,...,e_d$ are orthonormal vectors in $\rr^d$ and $\eps\in (0,1/100)$. Then
	\[
	\sfd_H( \Gamma_1\cap B_{4}(0),\Gamma_2\cap B_{4}(0))\le C\eps.
	\]
\end{lemma}

The following result is also standard and can be easily proved applying Gram-Schmidt orthogonalization procedure together with a  straightforward computation in coordinates (see for example Lemma 7.11 in \cite{snow}).
\begin{lemma}\label{lem:isometrylemma}
	For every $n \in \mathbb{N}$ exists a constant $C=C(n)>0$ such that the following holds. Let  $f:{B^{\rr^n}_{r}(0)} \to  \rr^d$, $d> n$, be an $(\eps r)$-isometry (see Section \ref{sec:proofmain}) with $\eps\in(0,1)$, then there exists an  isometry $I: \rr^n \to \rr^d$ such that $I(0)=f(0)$ and satisfying
	\begin{equation}\label{nn}
		|I-f|\le C\sqrt{\eps} r, \quad \text{ on $B^{\rr^n}_r(0)$}.
	\end{equation}
\end{lemma}

We now define a notion of distance between affine planes in $\rr^d.$
\begin{definition}
	Let $\Gamma_1,\Gamma_2$ be two affine $n$-dimensional planes,  we put
	\[\sfd(\Gamma_1,\Gamma_2)\coloneqq\sfd_H(\tilde \Gamma_1\cap B_1(0),\tilde\Gamma_2\cap B_1(0)),\]
	where $\tilde \Gamma_i$ is the $n$  dimensional plane  parallel to $\Gamma_i$ and passing through the origin.
\end{definition}
Notice that the function $\sfd$ just defined clearly satisfies $\sfd(\Gamma_1,\Gamma_3)\le \sfd(\Gamma_1,\Gamma_2)+\sfd(\Gamma_2,\Gamma_3)$.

The following elementary lemma says that if two affine planes are sufficiently close with respect to the distance $\sfd$, then they are not orthogonal to each other.
\begin{lemma}\label{nonorth}
	Let $\Gamma_1,\Gamma_2$ two $n$-dimensional affine planes in $\rr^d$ such that $\sfd(\Gamma_1,\Gamma_2)< 1$. Write $\Gamma_i$ as $p_i+V_i$ where $p_i \in \rr^d$ and $V_i$ is a $n$-dimensional subspace of $\rr^d.$ Then
	\begin{equation}
		V_1^\perp \oplus V_2=\rr^d.
	\end{equation}
	In particular for every $p \in \Gamma_1$ there exists $q \in \Gamma_2$ such that $\Pi(q)=p$, where $\Pi$ is the orthogonal projection onto $\Gamma_1.$
\end{lemma}
\begin{proof}
	It's enough to prove that $V_1^\perp \cap V_2=\{0\}.$ Suppose $v \in V_1^\perp \cap V_2$, then  we can regard $V_1,V_2$ as affine planes through the origin and parallel to $\Gamma_1,\Gamma_2$. Therefore by hypothesis
	\[ |v|=\sfd(v,V_1)\le\sfd(\Gamma_1,\Gamma_2)|v|\]
	and thus $v=0.$
\end{proof}

The following simple technical result will be the main tool for the proof of Theorem \ref{thm:precise}.
\begin{lemma}\label{lem:pitagora}
	Let $\Gamma_1,\Gamma_2$ two $n$-dimensional affine planes in $\rr^d$. Then for any  $x\in \Gamma_1$ and any $y \in \rr^d$ (different from $x$)
	\begin{equation*}
		|x-y|^2\le |\Pi(x)-\Pi(y)|^2+|x-y|^2\left (\sfd(\Gamma_1,\Gamma_2)+\frac{\sfd(y,\Gamma_1)}{|x-y|}\right )^2,
	\end{equation*}
	where  $\Pi$ denotes the orthogonal projection onto $\Gamma_2$.
\end{lemma}
\begin{proof}
	Let $\alpha\coloneqq\sfd(\Gamma_1,\Gamma_2)$. Up to  translating both the plane $\Gamma_1$ and the points $x,y$ by the vector $\Pi(x)-x$, we can suppose $x \in \Gamma_2$ and $x=0$. Let now $p$ be the orthogonal projection of $y$ onto $\Gamma_1$. Since both $\Gamma_1$ and $\Gamma_2$  contain the origin, we have that
	\[\sfd(p,\Gamma_2)\le \sfd_H(\Gamma_2\cap B_{|p|}(0),\Gamma_1\cap B_{|p|}(0))\le |p|\alpha\le |y|\alpha.\]
	Therefore $d(y,\Gamma_2)\le\sfd(y,p)+\sfd(p,\Gamma_2)=d(y,p)+\sfd(p,\Gamma_2)\le \sfd(y,\Gamma_1)+|y|\alpha.$
	Then by Pythagoras' theorem
	\[ |y|^2=|\Pi(y)-y|^2+|\Pi(y)|^2=\sfd(y,\Gamma_2)^2+|\Pi(y)|^2\le(\sfd(y,\Gamma_1)+|y|\alpha)^2+|\Pi(y)-\Pi(x)|^2, \]
	since $\Pi(x)=0$. This concludes the proof.
\end{proof}

We conclude with two results about the numbers $\alpha(x,r)$ and $\beta(x,r)$ (recall their definition in \eqref{eq:alfadef}, \eqref{eq:betadef}), which are well known in the literature. The first one shows that there exists a plane which  realizes $\beta(x,r)$ (i.e. that minimizes \eqref{eq:betadef}) and at the same time almost realizes $\alpha(x,r)$. The second is a classical tilting estimates, which says that the orientation of such realizing plane do  not vary too much from scale to scale and between points close to each other.

\begin{prop}[Realizing plane]\label{prop:realizing planes}
	For every $n \in \mathbb{N} $ there exists a constant $C=C(n)\ge 1$ such that the following holds. Let $S \subset \rr^d$ with $d>n$, let $x \in S$ and $r>0$ be such that $\alpha(x,r)\le 1/100,$  then there exists an $n$-dimensional affine plane $\Gamma_x^r$  that realizes $\beta(x,r)$ and such that
	\[
	r^{-1}d_H(S\cap B_{r}(x),\Gamma_x^r\cap B_{r}(x))\le C\alpha(x,r).
	\]
\end{prop}
\begin{proof}
	The existence of two planes $\Gamma$ and $\Gamma'$ that realize respectively $\beta(x,r)$  and $\alpha(x,r),$ follows by compactness. Without loss of generality we can assume that $x=0$ and $r=1$. Since $0 \in\Gamma'$ there exist orthonormal vectors $e_1,...,e_n \in \Gamma'$ and points $x_1,...,x_n \in S\cap B_1(0)$ such that $|x_i-e_i|\le \alpha(0,1)$, $i=1,...,n.$ Moreover there exist points $y_0,y_1,...,y_n \in \Gamma$ such that $|y_i-x_i|,|y_0|\le \beta(0,1)\le \alpha(0,1)$, $i=1,...,n.$ In particular $|y_i-y_0-e_i|\le 4\alpha(0,1)$, $i=1,...,n$ and we can apply Lemma \ref{lem:frame close} to deduce that $\sfd_H(\Gamma\cap B_1(0),\Gamma'\cap B_1(0))\le C\alpha(0,1),$ which concludes the proof.
\end{proof}

\begin{prop}[Tilting estimate]\label{prop:tilt}
	For any $n \in \mathbb{N}$ there exist $\alpha=\alpha(n),\, C=C(n)>0$ such that the following holds. Let $S \subset \rr^d$ with $n>d$ and let $r>0$ and $x,y \in S$ be such that $\alpha(x,r),\alpha(y,r)\le \alpha(n)$ and $|x-y|<\frac 14r$. Then it holds 
	\[
	\begin{split}
		&\sfd(\Gamma_x^r,\Gamma_x^{2r})\le C(\beta(x,r)+\beta(x,2r)),\\
		&\sfd(\Gamma_x^r,\Gamma_y^r)\le C(\beta(x,r)+\beta(y,r)),
	\end{split}
	\]
	for any choice of realizing planes $\Gamma_x^r,\Gamma_y^r,\Gamma_x^{2r}$ (as given by Prop. \ref{prop:realizing planes}).
\end{prop}
\begin{proof}
	We prove only the second, since the first is analogous.
	
	As usual, the scaling and translation invariant nature of the statement allows us to assume that $r=1$ and $x$ to be the origin. Then there exist orthonormal vectors $e_1,...,e_n \in \Gamma_x^1$ and points $x=x_0,x_1,...,x_n\in S\cap B_1(0)$ such that $|x_i-1/2e_i|\le C(n)\alpha(x,1),$  $i=1,...,n.$  Moreover (if $\alpha(x,1)$ is small enough) $x_i \in B_1(y)$, hence there exist points $y_0,...,y_n\in \Gamma_y^{1}$ such that $|x_i-y_i|\le \beta(y,1),$   $i=0,...,n.$ Finally there exist points $z_1,...,z_n \in \Gamma_x^1$ such that $|z_i-x_i|\le \beta(x,1),$ $i=1,...,n.$ Putting all together we have $|y_i-y_0-1/2e_i|\le C\alpha(x,1)+2\beta(y,1)+\beta(x,1)$, $i=1,...,n$ and $\sfd(y_i,\Gamma_x^1)\le \beta(0,1)+\beta(y,1),$ $i=0,...,n$, hence (if $\alpha(x,1),\alpha(y,1)$ are small enough) we can apply Lemma \ref{lem:frame close} to deduce that $\sfd_H(\Gamma_x^1\cap B_1(0),\Gamma_y^1\cap B_1(0))\le C(\beta(x,1)+\beta(y,1))$. Observing that  $ \sfd(x,\Gamma_y^1)\le \beta(x,1)+\beta(y,1) $ and recalling that $x$ is the origin concludes the proof.
\end{proof}

\section{Proof of the main theorems: $\a\le \alpha^2$ and $\b \le \beta^2$}\label{sec:proofmain}
Both Theorem \ref{thm:mainA} and Theorem \ref{thm:mainB} will be deduced as corollaries of the following more precise result.
\begin{theorem}\label{thm:precise}
	For every $n \in \mathbb{N}$ there exist $C=C(n)>0,\eps=\eps(n)>0$ such that the following holds. Let $i \in \mathbb{N}_{0}$, $S \subset \rr^d$ with $d>n$ and assume that $\alpha_j\le \eps$ for every $j\ge i-2 $ (where $\alpha_j$ are as in \eqref{def:ai}), then
	\begin{equation}\label{eq:precisea}
		\a(x,2^{-i})\le C\left(\sup_{j \in \mathbb{N}_0} \frac{\big(\beta_{i-2}+...+\beta_{i+j}\big)^2}{2^j}\right)\vee C\alpha_i^2,\quad \forall \, x \in S, \, |x|\le1-2^{-i}, \tag{A}
	\end{equation}
	\begin{equation}\label{eq:preciseb}
		\b(x,2^{-i})\le C\sup_{j \in \mathbb{N}_0} \frac{\big(\beta_{i-2}+...+\beta_{i+j}\big)^2}{2^j},\quad \forall \, x \in S, \, |x|\le1-2^{-i}\tag{B},
	\end{equation}
where   $\beta_j$ are as in \eqref{eq:bi}.
\end{theorem}
Inequalities \eqref{eq:precisea} and \eqref{eq:preciseb} should be thought as weak versions of the formal inequalities ``$\a(x,r)\le C\alpha(x,r)^2$" and ``$\b(x,r)\le C\beta(x,r)^2$" that are not true in general since, as we saw in the introduction, \eqref{eq:trivial bound} and \eqref{eq:trivial bound 2} cannot be improved.

\begin{proof}[Proof of Thoerem \ref{thm:mainA} and Thoerem \ref{thm:mainB}, given Theorem \ref{thm:precise}]
Let $\eps\in(0,1/2)$ and $\bar i \in \mathbb{N}$ be as in the hypotheses of Thoerem \ref{thm:mainA} and Thoerem \ref{thm:mainB}. Since  $\eps>2^{-\bar i}$, from Theorem \ref{thm:precise} and the definition of the numbers $\a_i$ we have
\begin{equation}\label{eq:step1}
	\begin{split}
			(\a_i)^\lambda\le \sup_{\substack{  {x \in S,}\\ {|x|\le1-2^{-i}}}}\a(x,2^{-i})^\lambda&\le C\left(\sup_{j \in \mathbb{N}_0} \frac{\big(\beta_{i-2}+...+\beta_{i+j}\big)^{2\lambda}}{2^{\lambda j}}\right)\vee C\alpha_i^{2\lambda}\\
						&\le C\alpha_i^{2\lambda}+C\sum_{j \ge 0} \frac{(j+3)^{2\lambda-1}\vee 1}{2^{\lambda j}} (\beta_{i-2}^{2\lambda}+\dots+ \beta_{i+j}^{2\lambda}), \quad \forall i \ge \bar i.
\end{split}
\end{equation}
An analogous estimate holds for $\b_i$, $\forall i \ge \bar i$. Recalling that $\beta_i\le \alpha_i$ we obtain
\begin{align*}
	\sum_{i\ge \bar i} \a_i^\lambda &\le C\sum_{i\ge \bar i} \alpha_i^{2\lambda}+C\sum_{i\ge \bar i} \sum_{j \ge 0} \frac{(j+3)^{2\lambda-1}\vee 1}{2^{\lambda j}} (\alpha_{i-2}^{2\lambda}+\dots+ \alpha_{i+j}^{2\lambda})\\
	&\le C\sum_{i\ge \bar i} \alpha_i^{2\lambda} + C\sum_{j \ge 0} \frac{(j+3)^{2\lambda-1}\vee 1}{2^{\lambda j}} \sum_{i\ge \bar i} (\alpha_{i-2}^{2\lambda}+\dots+ \alpha_{i+j}^{2\lambda})\\
	&\le  C\sum_{i\ge \bar i} \alpha_i^{2\lambda} +C\left(\sum_{j \ge 0} \frac{(j+3)^{2\lambda}\vee (j+3)}{2^{\lambda j}} \right)\left( \sum_{i\ge \bar i-2} \alpha_i^{2\lambda} \right),
\end{align*}
which proves \eqref{eq:mainA precise}. The exact same computations yields also \eqref{eq:mainB precise}.
\end{proof}

Before passing to the proof of Theorem \ref{thm:precise} we recall the definition of \emph{$\delta$-isometry} and how it can be used to estimate the Gromov-Hausdorff distance. Given two metric spaces $(\X_i,\sfd_i)$, $i=1,2$ and a number $\delta>0$ we say that a map $f: \X_1 \to \X_2$ is a $\delta$-isometry if $|\sfd_2(f(x),f(y))-\sfd_1(x,y)|<\delta$ for every $x,y \in X_1.$ It holds that
\[
\sfd_{GH}((\X_1,\sfd_1),(\X_2,\sfd_2))\le 2\inf \{\delta>0 \ : \ \exists\,  \delta\text{-isometry } f : \X_1 \to \X_2, \text{ with $f(X_1)$ $\delta$-dense in $X_2$}\},
\]
see for example \cite{burago} for a proof.

\begin{proof}[Proof of Theorem \ref{thm:precise}]
Observe that  it is sufficient to consider the case $x=0$ and $i=0$ for both \eqref{eq:precisea} and \eqref{eq:preciseb}, since the conclusion then follows by translating and scaling.

	We define 
	\begin{equation}\label{alpha}
		\theta\coloneqq  C^2\sup_{j \in \mathbb{N}_0} \frac{\big(\beta_{-2}+...+\beta_{j}\big)^2}{2^j},
	\end{equation}
\begin{equation}\label{alpha}
	\theta'\coloneqq  \max(\theta,C\alpha_0^2),
\end{equation}
where $C$ is a big enough constant depending only on $n$, to be determined later.
	Before proceeding we make the following observation 
	\begin{equation}\label{goodscaling}
		C^2(\beta_{-2}+...+\beta_{j})^2> \lambda>0 \, \implies \theta> \frac{2\lambda}{  2^j}.
	\end{equation}
Along the proof, for a given $x\in S$ and $r>0$ we will denote by $\Gamma_x^r$ one of the realizing planes given by Proposition \ref{prop:realizing planes} (the choice of the particular plane is not relevant).

	{\bf Proof of \eqref{eq:preciseb}:} Let $\Pi$ be the orthogonal projection onto $\Gamma_0^1$.  It is sufficient to show that
	\begin{equation}\label{eq:tetaisom}
		\Pi : S\cap B_{1}(0) \to \Gamma_{0}^1\cap B_{1}(0), \quad \text{is a $\theta$-isometry},
	\end{equation}
with respect to the Euclidean distance. 

	Choose $x,y \in S\cap B_{1}(0)$ distinct and observe that there exists a unique integer $j\ge 0$ such that 
	\begin{equation}\label{scalenj}
		\frac{1}{2^{j}}\le|x-y|< \frac{1}{2^{j-1}}.
	\end{equation}
Applying Proposition \ref{prop:tilt}  multiple times (assuming $\alpha_j\le \alpha(n)$ for every $j\ge i-2$, with $\alpha(n)$ as in the statement of Prop. \ref{prop:tilt}) we have
	\begin{align}
		\sfd(\Gamma_0^1,\Gamma^{2^{-j+1}}_x)&\le\sfd(\Gamma^{1}_{0},\Gamma^{2}_{0})+\sfd(\Gamma^{2}_{0},\Gamma^{2^{2}}_{0})+\sfd(\Gamma^{2^{2}}_{0},\Gamma^{2^{2}}_x) \nonumber+\sfd(\Gamma^{2^{2}}_x,\Gamma^{2^{1}}_x)+...+\sfd(\Gamma^{2^{-j+2}}_x,\Gamma^{2^{-j+1}}_x)\\\nonumber
		&\le D(\beta(0,0)+\beta(0,2)+\beta(0,2^{2})+\beta(x,2^{2})+...+\beta(x,2^{-j+1}))\\
		&\le D(\beta_{-2}+...+\beta_{j-1}),\label{planebound}
	\end{align}
for some constant  $D$ depending only on $n$.
	We consider now two cases, when $(D+4)(\beta_{-2}+...+\beta_{j-1})>1$ or the opposite. In the first case, assuming that $C\ge D+4$, from \eqref{goodscaling} we have $\theta\ge \frac{2}{2^j}$ and therefore
	\[||\Pi(x)-\Pi(y)|-|x-y||\le |x-y| \le \frac{2}{2^{j}}< \theta,\]
	that is what we wanted. Hence we can suppose that $(D+4) (\beta_{-2}+...+\beta_{j-1})\le 1.$
	Since from \eqref{scalenj} it holds that $|x-y|\ge 2^{-j}$,  we have that 
	\begin{equation}\label{eq:boh}
		\sfd(y,\Gamma^{2^{-j+1}}_x)\le 4\beta(x,2^{-j+1})|x-y|.
	\end{equation}
	 We can now apply Lemma \ref{lem:pitagora} to the planes $\Gamma_0^1,\Gamma^{2^{-j+1}}_x$, that coupled with  \eqref{planebound} and \eqref{eq:boh} gives
	\begin{align*}
		|\Pi(x)-\Pi(y)|&\ge |x-y|\sqrt{1-(D(\beta_{-2}+...+\beta_{j-1})+4\beta_{j-1})^2}.
	\end{align*}
	Hence
	\begin{align*}
		|x-y|-|\Pi(x)-\Pi(y)|&\le  |x-y|\left(1- \sqrt{1-\left((D+4)(\beta_{-2}+...+\beta_{j-1})	\right)^2}\right).
	\end{align*}
	Thanks to the assumption  $(D+4)(\beta_{-2}+...+\beta_{j-1})\le 1$, we can  use  the elementary the inequality $1-\sqrt{1-x}\le x$, valid for $0\le x\le 1$, to finally obtain
	\begin{align*}
			||x-y|-|\Pi(x)-\Pi(y)||&\le |x-y| ((D+4)(\beta_{-2}+...+\beta_{j-1}))^2\\
		&\overset{\eqref{scalenj} }{\le} \frac{((D+4)(\beta_{-2}+...+\beta_{j-1}))^2}{2^{j-1}}\le \theta,
	\end{align*}
	where we have used the definition of $\theta$ and assuming  $C\ge 2(D+4)$. This concludes the proof of  \eqref{eq:tetaisom} and thus the proof of \eqref{eq:preciseb}.\\
	
{\bf Proof of \eqref{eq:precisea}:} In view of \eqref{eq:tetaisom}, we only need to show that $\Pi$ is also $\theta'$-surjective.

	\emph{Claim:} Let $C'=C'(n)\ge 1$ be the constant given by Proposition \ref{prop:realizing planes}. For every $p \in \Gamma_0^1\cap B_{1}(0)$ and $x \in S\cap B_{1}(0)$ such that $C'\alpha_{0}\ge |p-\Pi(x)|\ge \theta' $  it holds
	\[ B_{\frac{3}{4}|p-\Pi(x)|}(p)\cap \Pi(S\cap B_{1}(0)) \neq \emptyset.\] 
	
	Before proving the claim, we show that it implies that $\Pi$ is $\theta'$-surjective. Indeed suppose it is not, i.e. there exists $p \in \Gamma_0^1 \cap B_{1}(0)$ such that 
	$$ R\coloneqq \sup\{r \, |\, B_{r}(p)\cap \Pi( S\cap B_{1}(0))= \emptyset\}\ge \theta'.$$
	Since $d_H(\Gamma_0^1\cap B_{1}(0),S\cap B_{1}(0))\le C'\alpha_{0}$  (recall that $\Gamma_0^1$ was chosen as a realizing plane as given by Prop. \ref{prop:realizing planes}), there exists $x \in S\cap B_{1}(0)$ such that $|x-p|\le C'\alpha_{0}$ and in particular $|\Pi(x)-p|\le C'\alpha_{0}$. Therefore $R \le C'\alpha_0$. This implies, from the definition of $R$, that there exists a point $x'\in S\cap B_{1}(0)$ such that $\theta'\le R\le |\Pi(x')-p|\le \min(\frac{5}{4}R,C'\alpha_0)$.  
	However the Claim gives that $$\emptyset \neq B_{\frac{3}{4}|\Pi(x')-p|}(p) \cap \Pi(S\cap B_{1}(0))\subset B_{\frac{15}{16}R}(p) \cap \Pi(S\cap B_{1}(0)),$$ that contradicts the minimality of $R$.

	\emph{Proof of the Claim}: Set $R\coloneqq |p-\Pi(x)|.$ To make the proof more easy to follow we first explain the intuition behind it. The key idea is that near $x$ the set $S$ is distributed in a horizontal manner, near a plane passing through $x$. We can then move along this plane towards $p$ and thus find a point $y$ in $S\cap B_{1}(0)$ such that $|\Pi(y)-p|\sim \frac{R}{2}$. However, since $p$ can be near the boundary of $B_{1}(0)$, in this movement we might go outside the ball $B_1(0)$. To avoid this issue we consider a point $q$ such that $|p-q|\sim \frac{R}{2}$ but placed radially towards the origin and then find a point $y$ (using the idea  described above of moving horizontally near $x$) that projects near $q$.  \\
	Start by noticing that (if $\alpha_0$ is small enough w.r.t.\ $n$) $R\le C'\alpha_0 < 1/4.$ Therefore there exists a  unique integer $j\ge 2$ such that 
	\begin{equation}\label{scalnj2}
		\frac{1}{2^{j+1}}\le R < \frac{1}{2^{j}}.
	\end{equation}
	Since by assumption $\theta\le \theta'\le R \le 1/2^{j}$, from \eqref{goodscaling} we have $(C(\beta_{-2}+...+\beta_{j}))^2\le 1/2$. Define now the point $q \in \Gamma_0^1 \cap B_{1}(0)$ as
	\[q=p-\frac{p}{|p|}\frac{R}{2}.\]
	Then
	\begin{equation}\label{qnorm}
		|q|=\left ||p|-\frac{R}{2}\right |\le 1-\frac{R}{2},
	\end{equation}
	indeed $|p|< 1$ and $R<1.$ Moreover $|p-q|=R/2$ and $|q-\Pi(x)|\le |p-\Pi(x)|+|p-q|=3/2R.$
	Consider  the plane $\Gamma_{x}^{2^{-j}}$, arguing as in \eqref{planebound} we can show that
	\[\sfd(\Gamma_0^1,\Gamma_{x}^{2^{-j}})\le C(\beta_{-2}+...+\beta_{j})<1, \]
	provided $C$ is big enough.
	Then by Proposition \ref{nonorth} there exists a point $e \in \Gamma_{x}^{2^{-j}}$ such that $\Pi(e)=q$. Applying Lemma \ref{lem:pitagora}  we obtain
	\begin{align*}
		|e-x|^2\le|q-\Pi(x)|^2+|e-x|^2/2
	\end{align*}
	that implies $|e-x|\le \sqrt 2 |q-\Pi(x)|\le 3R \le 1/{2^{j-2}}  $. Therefore  there exists $y \in S\cap B_{2^{-j+2}}(x)$ such that $|y-e|\le \alpha_{j-2}2^{-j+2}< R/4 $ (provided $\alpha_{j-2}$ is small enough). Thus
	\[|\Pi(y)-p|\le |\Pi(y)-\Pi(e)|+|p-q|\le|y-e|+R/2< 3/4R, \]
	that means $\Pi(y) \in B_{\frac{3}{4}R}(p). $ It remains to prove that $y \in B_{1}(0).$ First we observe that from \eqref{scalnj2} and the assumption $R\le C'\alpha_0 $ we have
	\[|y-x|\le \frac{4}{2^{j}}\le 8R\le 8C' \alpha_0.\]
	Hence, since $x \in B_1(0),$
	\[\sfd(y,\Gamma_0^1)\le |x-y|+\sfd(x,\Gamma_0^1)\le 9C'\alpha_0. \]
	From previous computations we know that that $|\Pi(y)-q|=|\Pi(y)-\Pi(e)|\le |y-e|\le R/4$, therefore from \eqref{qnorm} $|\Pi(y)|\le |q|+R/4\le 1-R/4.$ From Pythagoras Theorem we obtain
	\begin{align*}
		|y|^2= |\Pi(y)|^2+\sfd(y,\Gamma_0^1)^2&\le  \left (1-\frac{R}{4}\right )^2+(9C')^2 \alpha_0^2=\\
		&=1+R\left(\frac{R}{16}+\frac{(9C')^2\alpha_0^2}{R}-\frac{1}{2} \right).
	\end{align*}
	Thus to conclude it is enough to show that 
	\[\frac{R}{16}+\frac{(9C')^2\alpha_0^2}{R}<\frac{1}{2}.\]
	Since by assumption $R\ge \theta'$ and by definition $\theta' \ge C\alpha_0^2$, we deduce that $\frac{\alpha_0^2}{R}\le \frac{1}{ C}$. Therefore recalling that $R<1$, the above inequality is satisfied as soon as $C<4(9C')^2$. This concludes the proof.
\end{proof}
	
\section{Converse inequalities : $\alpha^2 \le \a$, $\beta^2\le \b$}
As explained in the introduction and in Section \ref{sec:proofmain}, the inequalities ``$\a(x,r)\le C\alpha(x,r)^2$" and ``$\b(x,r)\le C\beta(x,r)^2$" are not true in general, but hold only in their weaker formulations contained in Theorem \ref{thm:precise}. In this final section we will prove that the opposite inequality $\alpha(x,r)^2\le C\a(x,r)$ do hold in general, together with a weaker version of $\beta(x,r)^2\le C \b(x,r)$.

\begin{prop}
	Let $S\subset \rr^d$ and $n \in \mathbb{N}$ with $n<d$. Then 
		\[
	\alpha(x,r)^2\le C\a(x,r), \quad \text{for every }x\in S, r>0,
	\]
	where $\alpha(x,r)$ is as in \eqref{eq:alfadef} and $C=C(n)>0.$
\end{prop}
\begin{proof}
	Straightforward from Lemma \ref{lem:isometrylemma}
\end{proof}

Given a set of $n+1$ points $x_0,...,x_n\in \rr^d$ we denote by ${\sf Vol}_{n}(x_0,...,x_{n})$  the volume of the $n$-dimensional simplex with vertices $x_0,...,x_n.$ It is well know  that for every $n \in \mathbb{N}$ there exists a polynomial $P_n: \rr^{(n+1)n/2}\to \rr$ such that
\begin{equation}\label{eq:menger}
	{\sf Vol}_{n}(x_0,...,x_{n})^2=P_n(\{|x_i-x_j|^2\}_{0\le i<j\le n}),
\end{equation}
see for example \cite[$\S$ 40]{blu} for a proof.

The following results states that $\beta(x,r)\le C\sqrt{\b(x,r)}$, provided $B_r(x)\cap S$ contains $n$ points which are `sufficiently independent' in the sense that they span a simplex with large volume.
\begin{prop}
	Let $S\subset \rr^d$ and $n \in \mathbb{N}$ with $n<d$. Then
	\[
\beta(x,r)\le C\left(\frac{\sqrt{\b(x,r)}}{V_n}\wedge V_n^{\frac{1}{n}}\right),	 \quad \text{for every }x\in S, r>0,
\]
where $V_n=\sup_{\{x_i\}_{i=0}^n\subset S \cap B_r(x)} r^{-n}{\sf Vol}_{n}(x_0,...,x_n)$, 	 $\beta(x,r)$ is as in \eqref{eq:betadef} and $C=C(n)>0.$  In particular, if $\alpha(x,r)<1/8$ ($\alpha(x,r)$ is as in \eqref{eq:alfadef}), then
\[
\beta(x,r)\le 100C\sqrt{\b(x,r)},	 \quad \text{for every }x\in S, r>0,
\]
\end{prop}
\begin{proof}
	After a rescaling we can consider only the case $r=1$. We can also suppose that $V_n>0,$ otherwise there is nothing to prove. 	
	Fix $\eps>0$. There exists a map $f: S\cap B_1(0)\to B^{\rr^n}_1(0)$ that is a $(\b(x,1)+\eps)$-isometry. Moreover there exist  $\{x_i\}_{i=0}^n\subset S \cap B_1(x)$ such that ${\sf Vol}_{n}(x_0,...,x_n)>V_n-\eps$. Let $x_{n+1} \in S\cap B_1(x)$  be arbitrary and observer that, since $f(x_0),...,f(x_n),f(x_{n+1}) \in \rr^n$, we must have ${\sf Vol}_{n+1}(f(x_0),...,f(x_n),f(x_{n+1}))=0$. From \eqref{eq:menger} and the fact that $P_{n+1}$ is locally Lipschitz, it follows that 
	\[
{\sf Vol}_{n+1}(x_0,...,x_n,\bar x)^2\le C(n)\sup_{0\le i<j\le n}||f(x_i)-f(x_j)|^2-|x_i-x_j|^2|\le4C(n)(\b(x,1)+\eps). 
\]
Therefore, denoted by $\Gamma$ the $n$-dimensional plane spanned by $x_0,...,x_n$, it holds
\[
\sfd(x_{n+1},\Gamma)=\frac{{\sf Vol}_{n+1}(x_0,...,x_n,x_{n+1})}{{\sf Vol}_n(x_0,...,x_n)}\le C(n)\frac{\sqrt{\b(x,1)+\eps}}{V_n-\eps}.
\]
Moreover it is clear that there exists a constant $C'(n)>0$ such that ${\sf Vol}_{n+1}(x_0,...,x_n,x_{n+1})\le C'V_n^{\frac{n+1}{n}}$.   From the arbitrariness of $x_{n+1}\in S\cap B_1(x)$ and $\eps>0$ the conclusion follows.
\end{proof}

\end{document}